\documentclass[12pt]{amsart}
\usepackage{amsmath,amsthm,amsfonts,amssymb,mathrsfs}
\date{\today}

\usepackage{color}
\input xymatrix
\xyoption{all}

\usepackage[colorlinks,plainpages,citecolor=magenta, linkcolor=blue, backref]{hyperref}

\usepackage{hyperref}

\newtheorem{theorem}{Theorem}

\newtheorem{proposition}[theorem]{Proposition}
\newtheorem{corollary}[theorem]{Corollary}
\newtheorem{lemma}[theorem]{Lemma}
\theoremstyle{definition}
\newtheorem{example}[theorem]{Example}

\newtheorem{definition}[theorem]{Definition}

\begin{document}

\title[On locally compact shift-continuous topologies]{On locally compact shift-continuous topologies on semigroups $\mathscr{C}_{+}(a,b)$ and $\mathscr{C}_{-}(a,b)$ with adjoined zero}

\author{Oleg~Gutik}
\address{Faculty of Mathematics, National University of Lviv,
Universytetska 1, Lviv, 79000, Ukraine}
\email{oleg.gutik@lnu.edu.ua}



\begin{abstract}
Let $\mathscr{C}_{+}(p,q)^0$ and $\mathscr{C}_{-}(p,q)^0$ be the semigroups $\mathscr{C}_{+}(a,b)$ and $\mathscr{C}_{-}(a,b)$ with the adjoined zero.
We show that the semigroups $\mathscr{C}_{+}(p,q)^0$ and $\mathscr{C}_{-}(p,q)^0$ admit continuum many different Hausdorff locally compact shift-continuous topologies  up to topological isomorphism.
\end{abstract}

\maketitle

In this paper we shall follow the terminology of \cite{Carruth-Hildebrant-Koch=1983, Clifford-Preston=1961,  Engelking=1989, Ruppert=1984}.

By $\omega$ we denote the set of all non-negative integers. Throughout these notes we always assume that all topological spaces involved are Hausdorff.

\begin{definition}[\cite{Carruth-Hildebrant-Koch=1983, Ruppert=1984}]
Let $S$ be a non-void topological space which is provided with an associative multiplication (a semigroup operation) $\mu\colon S\times S\to S$, $(x,y)\mapsto \mu(x,y)=xy$. Then the pair $(S,\mu)$ is called
\begin{itemize}
  \item[$(i)$] a \emph{right topological} (\emph{left topological}) \emph{semigroup} if all interior left (right) shifts $\lambda_s\colon S\to S$, $x\mapsto sx$ ($\rho_s\colon S\to S$, $x\mapsto xs$), are continuous maps, $s\in S$;
  \item[$(ii)$] a \emph{semitopological semigroup} if the map $\mu$ is separately continuous;
  \item[$(iii)$] a \emph{topological semigroup} if the map $\mu$ is jointly continuous.
\end{itemize}
We usually omit the reference to $\mu$ and write simply $S$ instead of $(S,\mu)$. It goes without saying that every topological semigroup is
also semitopological and every semitopological semigroup is both a right and left topological semigroup.
\end{definition}

A topology $\tau$ on a semigroup $S$ is called:
\begin{itemize}
  \item a \emph{semigroup} topology if $(S,\tau)$ is a topological semigroup;
  \item a \emph{shift-continuous} topology if $(S,\tau)$ is a semitopological semigroup;
  \item an \emph{left-continuous} (\emph{right-continuous}) topology if $(S,\tau)$ is a left (right) topological semigroup.
\end{itemize}

The bicyclic monoid ${\mathscr{C}}(a,b)$ is the semigroup with the identity $1$ generated by two elements $a$ and $b$ subjected only to the condition $ab=1$. The semigroup operation on ${\mathscr{C}}(a,b)$ is determined as
follows:
\begin{equation*}\label{eq-1}
    b^ka^l\cdot b^ma^n=
    \left\{
      \begin{array}{ll}
        b^{k-l+m}a^n, & \hbox{if~} l<m;\\
        b^ka^n,       & \hbox{if~} l=m;\\
        b^ka^{l-m+n}, & \hbox{if~} l>m.
      \end{array}
    \right.
\end{equation*}

In \cite{Makanjuola-Umar=1997} Makanjuola and Umar study algebraic property of the following anti-isomorphic subsemigroups
\begin{equation*}
  \mathscr{C}_{+}(p,q)=\left\{q^ip^j\in\mathscr{C}(p,q)\colon i\leqslant j\right\}
  \quad \hbox{and} \quad
  \mathscr{C}_{-}(p,q)=\left\{q^ip^j\in\mathscr{C}(p,q)\colon i\geqslant j\right\},
\end{equation*}
of the bicyclic monoid. In the paper \cite{Gutik=2023} we prove that every Hausdorff left-continuous (right-continuous) topology on the monoid $\mathscr{C}_{+}(a,b)$  ($\mathscr{C}_{-}(a,b)$) is discrete and show that there exists a compact Hausdorff topological monoid $S$ which contains $\mathscr{C}_{+}(a,b)$  ($\mathscr{C}_{-}(a,b)$) as a submonoid. Also, in \cite{Gutik=2023} we constructed a non-discrete right-continuous (left-continuous) topology $\tau_p^+$ ($\tau_p^-$) on the semigroup $\mathscr{C}_{+}(a,b)$ ($\mathscr{C}_{-}(a,b)$) which is not left-continuous (right-continuous).

Later by $\mathscr{C}_{+}(p,q)^0$ and $\mathscr{C}_{-}(p,q)^0$  we denote the semigroups $\mathscr{C}_{+}(a,b)$ and $\mathscr{C}_{-}(a,b)$ with the adjoined zero.

In \cite{Gutik=2015} it is proved that every Hausdorff locally compact shift-continuous topology on the bicyclic monoid with adjoined zero is either compact  or discrete. This result was extended by Bardyla onto the $p$-polycyclic monoid \cite{Bardyla=2016} and graph inverse semigroups \cite{Bardyla=2018}, and by Mokrytskyi onto the monoid of order isomorphisms between principal filters of $\mathbb{N}^n$ with adjoined zero \cite{Mokrytskyi=2019}.
In \cite{Gutik-Khylynskyi=2022} the results of paper \cite{Gutik=2015} onto the monoid $\mathbf{I}\mathbb{N}_{\infty}$ of all partial cofinite isometries of positive integers with adjoined zero are extended. In \cite{Gutik-Mykhalenych=2023} the similar dichotomy was proved for so called bicyclic extensions $\boldsymbol{B}_{\omega}^{\mathscr{F}}$  when a family $\mathscr{F}$ consists of inductive non-empty subsets of~$\omega$.
Algebraic properties on a group $G$ such that if the discrete group $G$ has these properties, then every locally compact shift continuous topology on
$G$ with adjoined zero is either compact or discrete studied in \cite{Maksymyk=2019}. The above results are extended in \cite{Gutik-KhylynskyiM=2024} to the bicyclic extension $\boldsymbol{B}_{[0,\infty)}$ of the additive group of reals with adjoined zero (see \cite{Gutik-Pagon-Pavlyk=2011}) in the cases when on the semigroup $\boldsymbol{B}_{[0,\infty)}$ the usual topology, the discrete topology or  the topology determined by the natural partial order is defined.
Also, in \cite{Gutik-Maksymyk=2019} it is proved that the extended bicyclic semigroup $\mathscr{C}_\mathscr{\mathbb{Z}}^0$ with adjoined zero admits distinct $\mathfrak{c}$-many  shift-continuous topologies, however every Hausdorff locally compact semigroup topology on $\mathscr{C}_\mathscr{\mathbb{Z}}^0$ is discrete. In \cite{Bardyla=2023} Bardyla proved that a Hausdorff locally compact semitopological semigroup McAlister Semigroup $\mathcal{M}_1$ is either compact or discrete. However, this dichotomy does not hold for the McAlister Semigroup $\mathcal{M}_2$ and moreover, $\mathcal{M}_2$ admits continuum many different Hausdorff locally compact inverse semigroup topologies \cite{Bardyla=2023}.

In this paper we show that the semigroups $\mathscr{C}_{+}(p,q)^0$ and $\mathscr{C}_{-}(p,q)^0$ admit continuum many different Hausdorff locally compact shift-continuous topologies  up to topological isomorphism.

\begin{lemma}\label{lemma-3}
Every locally compact Hausdorff shift-continuous topology $\tau$ on the additive semigroup of non-negative integers $(\omega, +)$ is discrete.
\end{lemma}

\begin{proof}
Fix any $n_0\in\omega$. The Hausdorffness of the space $(\omega,\tau)$ implies that $n_0^{\downharpoonleft}=\{k\in\omega\colon k\leqslant n\}$ is a closed subset of $(\omega,\tau)$. Then $\omega\setminus n_0^{\downharpoonleft}$ is an open subset of $(\omega,\tau)$, and by Corollary~3.3.10 of \cite{Engelking=1989}, $\omega\setminus n_0^{\downharpoonleft}$ is locally compact, and hence, Baire. By Proposition~1.30 of \cite{Haworth-McCoy=1977} the space $\omega\setminus n_0^{\downharpoonleft}$ contains an isolated point $n_1$, which is isolated in $(\omega,\tau)$ because $\omega\setminus n_0^{\downharpoonleft}$ is an open subset of $(\omega,\tau)$. This and the condition $n_0<n_1$ imply that $n_0$ is an isolated point in $(\omega,\tau)$, because $n_0$ is the full preimage of $n_1$ under the continuous right shift $\rho_{n_1-n_0}\colon (\omega, +,\tau)\to (\omega, +,\tau)$, $i\mapsto i+(n_1-n_0)$. This completes the proof of the lemma.
\end{proof}

Later by $(\omega, +)^0$ we denote the additive semigroup of non-negative integers $(\omega, +)$ with adjoined zero. Without loss of generality we may assume that $(\omega, +)^0=\omega\cup\{\infty\}$ with the extended semigroup operation $n+\infty=\infty+n=\infty+\infty=\infty$ for all $n\in\omega$, i.e., $\infty$ is the zero of $(\omega, +)^0$.

\begin{proposition}\label{proposition-4}
Every Hausdorff locally compact shift-continuous topology on the semigroup $(\omega, +)^0$ is either compact or discrete.
\end{proposition}

\begin{proof}
Let $\tau_{\mathrm{lc}}$ be an arbitrary non-discrete Hausdorff locally compact shift-continuous topology on the semigroup $(\omega, +)^0$.  The Hausdorffness of $((\omega, +)^0,\tau_{\mathrm{lc}})$ implies that $\omega$ is an open subset of $((\omega, +)^0,\tau_{\mathrm{lc}})$. Then by Corollary~3.3.10 of \cite{Engelking=1989}, $\omega$ is locally compact, and by Lemma~\ref{lemma-3} is a discrete subspace of $((\omega, +)^0,\tau_{\mathrm{lc}})$.

Since all point from $\omega$ are open-and-closed subsets of the locally compact space $((\omega, +)^0,\tau_{\mathrm{lc}})$, there exists a base $\mathscr{B}_{\tau_{\mathrm{lc}}}(\infty)$ of the topology $\tau_{\mathrm{lc}}$ at the point $\infty$ which consists of compact-and-open subsets of $((\omega, +)^0,\tau_{\mathrm{lc}})$. Hence, for any $U,V\in \mathscr{B}_{\tau_{\mathrm{lc}}}(\infty)$ the set $U\setminus V$ is finite.

We state that for any $U\in \mathscr{B}_{\tau_{\mathrm{lc}}}(\infty)$ the set $\omega\setminus U$ is finite. Suppose to the contrary that there exists $U\in \mathscr{B}_{\tau_{\mathrm{lc}}}(\infty)$ the set $\omega\setminus U$ is infinite. The separate continuity of the semigroup operation in $((\omega, +)^0,\tau_{\mathrm{lc}})$ implies that there exists $V\in \mathscr{B}_{\tau_{\mathrm{lc}}}(\infty)$ such that $V\subseteq U$ and $1+V\subseteq U$. Since $\omega\setminus U$ is infinite, there exists a sequence $\{x_n\}_{n\in\omega}\subseteq U$ such that $1+x_i\neq U$ for any $i\in\omega$. This implies that $x_i\neq V$ for any $i\in\omega$, and hence, the set $U\setminus V$ is infinite, a contradiction. The obtained contradiction implies that $\tau_{\mathrm{lc}}$ is a compact topology on $(\omega, +)^0$.
\end{proof}

Later by $\tau_{\mathrm{lc}}$ we denote a Hausdorff locally compact shift-continuous topology on the semigroup $\mathscr{C}_{+}(p,q)^0$.

Since every Hausdorff shift-continuous topology on the semigroup $\mathscr{C}_{+}(p,q)$ is discrete (see \cite[Theorem~6]{Gutik=2023}), the following statements holds.

\begin{lemma}\label{lemma-5}
If $U$ and $V$ are any compact-and-open neighbourhoods of the zero in $(\mathscr{C}_{+}(p,q)^0,\tau_{\mathrm{lc}})$, then the set $U\setminus V$ is finite.
\end{lemma}

For any $i\in \omega$ we denote
\begin{equation*}
  \mathscr{C}_{+}^i(p,q)=\{b^ia^{i+s}\in \mathscr{C}_{+}(p,q)\colon s\in\omega\}.
\end{equation*}
The semigroup operation of $\mathscr{C}_{+}(p,q)$ implies that $\mathscr{C}_{+}^i(p,q)$ is a subsemigroup of $\mathscr{C}_{+}(p,q)$, and moreover, $\mathscr{C}_{+}^i(p,q)$ is isomorphic to the additive semigroup of non-negative integers $(\omega, +)$ for any $i\in\omega$ \cite{Gutik=2023}.

\begin{lemma}\label{lemma-6}
For any compact-and-open neighbourhood $U$ of the zero in $(\mathscr{C}_{+}(p,q)^0,\tau_{\mathrm{lc}})$ there exists $i\in\omega$ such that the set $U\cap \mathscr{C}_{+}^i(p,q)$ is infinite.
\end{lemma}

\begin{proof}
Suppose to the contrary that there exists a compact-and-open neighbourhood $U$ of the zero in $(\mathscr{C}_{+}(p,q)^0,\tau_{\mathrm{lc}})$ such that $\left|U\cap \mathscr{C}_{+}^i(p,q)\right|<\infty$ for any $i\in\omega$. Then there exists a sequence $\{i_j\}_{j\in\omega}\subseteq\omega$ such that $U\cap \mathscr{C}_{+}^{i_j}(p,q)\neq\varnothing$ for any $j\in\omega$. The separate continuity of the semigroup operation in $(\mathscr{C}_{+}(p,q)^0,\tau_{\mathrm{lc}})$ and local compactness of $\tau_{\mathrm{lc}}$ imply that there exists a compact-and-open neighbourhood $V$ of zero in $(\mathscr{C}_{+}(p,q)^0,\tau_{\mathrm{lc}})$ such that $V\subseteq U$ and $V\cdot a\subseteq U$. By the definition of the semigroup operation in $\mathscr{C}_{+}(p,q)$ we get that $\mathscr{C}_{+}^i(p,q)\cdot a\subseteq \mathscr{C}_{+}^i(p,q)$ for all $i\in\omega$. Since for any $j\in\omega$ the set $U\cap \mathscr{C}_{+}^{i_j}(p,q)$ is non-empty and finite, there exists maximal non-negative integer $s_j$ such that $b^{i_j}a^{i_j+s_j}\in U$ but $b^{i_j}a^{i_j+s_j}\notin V$. This implies that the set $U\setminus V$ is infinite, which contradicts Lemma~\ref{lemma-5}. The obtained contradiction implies the statement of the lemma.
\end{proof}

\begin{lemma}\label{lemma-7}
For any compact-and-open neighbourhood $U$ of the zero in $(\mathscr{C}_{+}(p,q)^0,\tau_{\mathrm{lc}})$ there exists $i_0\in\omega$ such that $\mathscr{C}_{+}^{i_0}(p,q)\cup\{0\}$ is a compact subset of $(\mathscr{C}_{+}(p,q)^0,\tau_{\mathrm{lc}})$.
\end{lemma}

\begin{proof}
By Lemma~\ref{lemma-6} for any compact-and-open neighbourhood $U$ of the zero in $(\mathscr{C}_{+}(p,q)^0,\tau_{\mathrm{lc}})$ there exists $i_0\in\omega$ such that the set $U\cap \mathscr{C}_{+}^{i_0}(p,q)$ is infinite. Since $\mathscr{C}_{+}(p,q)$ is a discrete subspace of $(\mathscr{C}_{+}(p,q)^0,\tau_{\mathrm{lc}})$,  $\mathscr{C}_{+}^{i_0}(p,q)\cup\{0\}$ is a closed subset of $(\mathscr{C}_{+}(p,q)^0,\tau_{\mathrm{lc}})$. By Corollary~3.3.10 of \cite{Engelking=1989}, $\mathscr{C}_{+}^{i_0}(p,q)\cup\{0\}$ is locally compact. Since the semigroup $\mathscr{C}_{+}^{i_0}(p,q)\cup\{0\}$ is isomorphic to the additive semigroup of non-negative integers with adjoined zero $(\omega, +)^0$, by Proposition~\ref{proposition-4} the semigroup $\mathscr{C}_{+}^{i_0}(p,q)\cup\{0\}$ is a compact subsemigroup of $(\mathscr{C}_{+}(p,q)^0,\tau_{\mathrm{lc}})$.
\end{proof}

\begin{lemma}\label{lemma-8}
 $\mathscr{C}_{+}^{i}(p,q)\cup\{0\}$ is a compact subset of $(\mathscr{C}_{+}(p,q)^0,\tau_{\mathrm{lc}})$ for any $i\in\omega$.
\end{lemma}

\begin{proof}
By Lemma~\ref{lemma-7} there exists $i_0\in\omega$ such that $\mathscr{C}_{+}^{i_0}(p,q)\cup\{0\}$ is a compact subset of $(\mathscr{C}_{+}(p,q)^0,\tau_{\mathrm{lc}})$. We fix an arbitrary $i\in\omega$. The semigroup operation in $\mathscr{C}_{+}(p,q)^0$ implies the following:
\begin{enumerate}
  \item[(1)] if $i<i_0$, then
       \begin{align*}
         a^{i_0-i}\cdot \mathscr{C}_{+}^{i_0}(p,q) &=\left\{a^{i_0-i}\cdot b^{i_0}a^{i_0+s}\colon s\in \omega\right\}= \\
          &=\left\{b^{i_0-(i_0-i)}a^{i_0+s}\colon s\in \omega\right\}= \\
          &=\left\{b^{i}a^{i_0+s}\colon s\in \omega\right\}= \\
          &=\mathscr{C}_{+}^{i}(p,q)\setminus\left\{b^ia^i,\ldots,b^ia^{i_0-1}\right\};
       \end{align*}

  \item[(2)] if $i>i_0$, then
       \begin{align*}
         b^ia^{i}\cdot \mathscr{C}_{+}^{i_0}(p,q)&=\left\{b^ia^{i}\cdot b^{i_0}a^{i_0+s}\colon s\in \omega\right\}= \\
          &=\left\{b^ia^{i}\cdot b^{i_0}a^{i_0}\cdot a^{s}\colon s\in \omega\right\}= \\
          &=\left\{b^ia^{i}\cdot a^{s}\colon s\in \omega\right\}= \\
          &=\left\{b^ia^{i+s}\colon s\in \omega\right\}= \\
          &=\mathscr{C}_{+}^{i}(p,q).
       \end{align*}
\end{enumerate}
Hence, if $i>i_0$, then $\mathscr{C}_{+}^{i}(p,q)\cup\{0\}$ is a compact subset of $(\mathscr{C}_{+}(p,q)^0,\tau_{\mathrm{lc}})$ as a continuous image of compact space $\mathscr{C}_{+}^{i_0}(p,q)\cup\{0\}$ under left shift $\lambda_{b^ia^{i}}\colon x\mapsto b^ia^{i}\cdot x$. Similar, in the case when $i<i_0$ we have that $\mathscr{C}_{+}^{i}(p,q)\cup\{0\}$ is compact, because it is the union of the finite family of compact subsets $\left\{a^{i_0-i}\cdot \mathscr{C}_{+}^{i_0}(p,q),\left\{b^ia^i\right\},\ldots,\left\{b^ia^{i_0-1}\right\}\right\}$.
\end{proof}


\begin{example}\label{example-9}
Let $\left\{x_i\right\}_{i\in\omega}$ be any non-decreasing sequence of non-negative integers. We define the topology $\tau_{\{x_i\}}$ on the semigroup $\mathscr{C}_{+}(p,q)^0$ in the following way. Put
\begin{equation*}
  U_{\{x_i\}}^n(0)=\left\{0\right\}\cup\left\{b^ka^{k+x_k+s}\colon k,s\in\omega \quad \hbox{and} \quad k+x_k+s>n\right\}.
\end{equation*}
We suppose that all points of the set $\mathscr{C}_{+}(p,q)$ are isolated in $(\mathscr{C}_{+}(p,q)^0,\tau_{\{x_i\}})$, and the family $\mathscr{B}_{\{x_i\}}(0)=\left\{U_{\{x_i\}}^n(0)\colon n\in\omega\right\}$ is the base of the topology $\tau_{\{x_i\}}$ at zero $0$ of the semigroup $\mathscr{C}_{+}(p,q)^0$.

It is obvious that the space  $(\mathscr{C}_{+}(p,q)^0,\tau_{\{x_i\}})$ is  Hausdorff and locally compact.

Next we show that the semigroup operation in $(\mathscr{C}_{+}(p,q)^0,\tau_{\{x_i\}})$ is separately continuous.

Suppose $b^ma^{m+x_m+s}\in U_{\{x_i\}}^n(0)$ and $m\leqslant n=i+j\geqslant i$. Then $m+x_m+s>n$, and hencem $i+j+x_m+s\geqslant n$, which implies that $b^ia^{i+j++x_m+s}\in U_{\{x_i\}}^n(0)$.

If $m>n=i+j\geqslant i$, then $b^ia^{i+j}\cdot b^ma^{m+x_m+s}=b^{m-j}a^{m+x_m+s}$. In the case when $m-j\leqslant n$ we have that $m+x_m+s\geqslant n$ and  $b^{m-j}a^{m+x_m+s}\in U_{\{x_i\}}^n(0)$. In the case when $m-j> n$ we have that
\begin{equation*}
m-j+x_{m-j}+s\leqslant m-j+x_{m}+s\leqslant m+x_{m}+s,
\end{equation*}
because $\left\{x_i\right\}_{i\in\omega}$ is a non-decreasing sequence of non-negative integers, and hence $b^{m-j}a^{m+x_m+s}\in U_{\{x_i\}}^n(0)$. Therefore, the inclusion $b^ia^{i+j}\cdot U_{\{x_i\}}^n(0)\subseteq U_{\{x_i\}}^n(0)$ holds for any $n\geqslant i+j$.

If $m\geqslant n$, then $m+x_m+s>n$, and hence, we have that
\begin{equation*}
  b^ma^{m+x_m+s}\cdot b^ia^{i+j}= b^ma^{m+x_m+s-i+i+j}= b^ma^{m+j+x_m+s}.
\end{equation*}
This implies that for any $n\geqslant i+j$ the following inclusion $ U_{\{x_i\}}^n(0)\cdot  b^ia^{i+j}\subseteq U_{\{x_i\}}^n(0)$ holds.

Therefore, the semigroup operation in $(\mathscr{C}_{+}(p,q)^0,\tau_{\{x_i\}})$ is separately continuous.
\end{example}

Since there exist continuum many non-decreasing sequence of non-negative integers in $\omega$, Lemma~\ref{lemma-8} and Example~\ref{example-9} imply the main theorem of this paper.

\begin{theorem}
On the semigroup $\mathscr{C}_{+}(p,q)^0$ $(\mathscr{C}_{-}(p,q)^0)$ there exist continuum many Hausdorff locally compact shift-continuous topologies up to topological isomorphism.
\end{theorem}

Since for any non-decreasing sequence of non-negative integers $\left\{x_i\right\}_{i\in\omega}$ in $\omega$ and any $n\in\omega$ the set $\mathscr{C}_{+}(p,q)^0\setminus U_{\{x_i\}}^n(0)$ is either finite or infinite, we get the following corollary.

\begin{corollary}
On the semigroup $\mathscr{C}_{+}(p,q)^0$ $(\mathscr{C}_{-}(p,q)^0)$ there exist exactly three Hausdorff locally compact shift-continuous topologies up to homeomorphism.
\end{corollary}

\end{document}